\documentclass[10pt]{article}
\usepackage[all]{xy}
\usepackage{amsfonts,amsmath,oldgerm,amssymb,amscd}
\newcommand{\ra}{\rightarrow}

\newcommand{\by}[1]{\stackrel{#1}{\ra}}

\newcommand{\surj}{\ra\!\!\!\ra}

\newcommand{\ol}{\overline}          \newcommand{\wt}{\widetilde}
\newcommand{\iso}{\by \sim}

\newtheorem{theorem}{Theorem}[section]
\newtheorem{proposition}[theorem]{Proposition}
\newtheorem{lemma}[theorem]{Lemma}

\newtheorem{corollary}[theorem]{Corollary}

\newcommand{\ga}{\alpha}          \newcommand{\gb}{\beta}

          \newcommand{\gl}{\lambda}

\newcommand{\CM}{\mbox{$\mathcal M$}}     \newcommand{\CN}{\mbox{$\mathcal N$}}

\newcommand{\CS}{\mbox{$\mathcal S$}}

\newcommand{\op}{\mbox{$\oplus$}}     
     
     \newcommand{\Hom}{\mbox{\rm Hom}}
       
\newcommand{\Um}{\mbox{\rm Um}}          \newcommand{\SL}{\mbox{\rm SL}}
          \newcommand{\ot}{\mbox{$\otimes$}}
\newcommand{\Aut}{\mbox{\rm Aut}}

\begin{document}

\begin{center}
{\large \bf Another definition of Euler class group of a Noetherian ring}\\
\vspace{.2in} {\large Manoj K. Keshari\footnote{Supported by BOYSCAST
        Fellowship $2008-09$ of Department of Science and Technology,
        India} and Satya Mandal$^2$}\\
\vspace{.1in} {\small $^1$Department of Mathematics, IIT Mumbai,
Mumbai - 400076, India \\$^2$Department of Mathematics, University of
Kansas, 1460 Jayhawk Blvd, Lawrence, KS 66045, USA \\
keshari@math.iitb.ac.in, mandal@math.ku.edu}
\end{center}

%%%%%%%%%%%%%%%%%%%%%%%%%%%%%%%%%%%%%%%%%%%%%%%%%%%%%%%%%
\section{Introduction}

All the rings are assumed to be commutative Noetherian and all the
modules are assumed to be finitely generated.

Let $A$ be a ring of dimension $n\geq 2$ and let $L$ be a projective
$A$-module of rank $1$. In \cite{BR}, Bhatwadekar and Raja Sridharan
defined the Euler class group of $A$ with respect to $L$, denoted by
$E(A,L)$. To the pair $(P,\chi)$, where $P$ is a projective $A$-module
of rank $n$ with determinant $L$ and $\chi : L \iso \wedge^n P$ is a
$L$-orientation of $P$, they attached an element of the Euler class
group, denoted by $e(P,\chi)$. One of the main result in \cite{BR} is
that $P$ has a unimodular element if and only if $e(P,\chi)$ is zero
in $E(A,L)$.

We will define the Euler class group of $A$ with respect to a
projective $A$-module $F=Q\op A$ of rank $n$, denoted by $E(A,F)$. To
the pair $(P,\chi)$, where $P$ is a projective $A$-module of rank $n$
and $\chi :\wedge^n F \iso \wedge^n P$ is a $F$-orientation of $P$, we
associate an element of the Euler class group, denoted by $e(P,\chi)$
and prove the following result: $P$ has a unimodular element if and
only if $e(P,\chi)$ is zero in $E(A,F)$. Note that when $F=L\op
A^{n-1}$, $E(A,F)$ is same as the Euler class group $E(A,L)$ defined
in \cite{BR}.

\section{Preliminaries}

Let $A$ be a ring and let $M$ be an $A$-module. For $m\in M$, we
define $O_M(m) = \{ \varphi(m) | \varphi \in \Hom_A(M,A) \}$. We say
that $m$ is {\it unimodular} if $O_M(m) = A$. The set of all
unimodular elements of $M$ will be denoted by $\Um(M)$.  Note that if
$P$ is a projective $A$-module and $P$ has a unimodular element, then
$P\iso P_1\op A$.

Let $P$ be a projective $A$-module.  Given an element $\varphi\in P^*$
and an element $p\in P$, we define an endomorphism $\varphi_p$ as the
composite $P\by \varphi A \by p P$. If $\varphi(p)=0$, then
$\varphi_p^2=0$ and hence $1+\varphi_p$ is a unipotent automorphism of
$P$.

By a {\it transvection}, we mean an automorphism of $P$ of the form
$1+\varphi_p$, where $\varphi(p)=0$ and either $\varphi$ is unimodular
in $P^*$ or $p$ is unimodular in $P$. We denote by $EL(P)$ the
subgroup of $\Aut(P)$ generated by all the transvections of $P$. Note
that $EL(P)$ is a normal subgroup of $\Aut(P)$.

Recall that if $A$ is a ring of dimension $n$ and if $P$ is a
projective $A$-module of rank $n$, then any surjection $\ga : P \surj
J$ is called a {\it generic} surjection of $P$ if $J$ is an ideal of $A$ of
height $n$.

The following result is due to Bhatwadekar and Roy (\cite{B-Roy}, Proposition 4.1)

\begin{proposition}\label{B-Roy4.1}
Let $B$ be a ring and let $I$ be an ideal of $B$. Let $P$ be a
 projective $B$-module. Then any element of $EL(P/IP)$ can be lifted
 to an automorphism of $P$.
\end{proposition}

We state some results from \cite{BR} for later use.

\begin{lemma}\label{3.0}
(\cite{BR}, Lemma 3.0) Let $A$ be a ring of dimension $n$ and let $P$
be a projective $A$-module of rank $n$. Let $\gl:P\surj J_0$ and
$\mu:P\surj J_1$ be two surjections, where $J_0$ and $J_1$ are ideals
of $A$ of height $n$. Then there exists an ideal $I$ of $A[T]$ of
height $n$ and a surjection $\ga(T): P[T] \surj I$ such that
$I(0)=J_0$, $I(1)=J_1$, $\ga(0)=\gl$ and $\ga(1)=\mu$.
\end{lemma}

For a rank $1$ projective $A$-module $L$ and $P'=L\op A^{n-1}$, the
following result is proved in (\cite{BR}, Proposition 3.1). Since the
same proof works in our case, we omit the proof.

\begin{proposition}\label{3.1}
Let $A$ be a ring of dimension $n\geq 2$ such that $(n-1)!$ is a unit
in $A$. Let $P$ and $P^\prime =Q\op A$ be projective $A$-modules of
rank $n$ and let $\chi : \wedge^n P \iso \wedge^n P^\prime$ be an
isomorphism. Suppose that $\ga(T) : P[T] \surj I$ be a surjection,
where $I$ is an ideal of $A[T]$ of height $n$. Then there exists a
homomorphism $\phi : P^\prime \ra P$, an ideal $K$ of $A$ of height
$\geq n$ which is comaximal with $(I\cap A)$ and a surjection $\rho(T)
: P^\prime[T] \surj I\cap KA[T]$ such that the following holds:

$(i)$ $\wedge^n( \phi)=u \chi$, where $u=1$ modulo $I\cap A$.

$(ii)$ $(\ga(0)\circ \phi)(P^\prime)=I(0)\cap K$.

$(iii)$ $(\ga(T)\circ \phi(T)) \ot A[T]/I = \rho(T)\ot A[T]/I$.

$(iv)$ $\rho(0)\ot A/K=\rho(1)\ot A/K$.

\end{proposition}

The following result is labeled as addition principle (\cite{BR}, Theorem 3.2).

\begin{theorem}\label{addition}
Let $A$ be a ring of dimension $n\geq 2$ and let $J_1,J_2$ be two
comaximal ideals of $A$ of height $n$.  Let $P=P_1\op A$ be a
projective $A$-module of rank $n$ and let $\Phi: P\surj J_1$ and
$\Psi:P\surj J_2$ be two surjections. Then, there exists a surjection
$\Theta : P\surj J_1\cap J_2$ such that $\Phi\ot A/J_1=\Theta \ot
A/J_1$ and $\Psi\ot A/J_2=\Theta\ot A/J_2$.
\end{theorem}

The following theorem is labeled as subtraction principle (\cite{BR}, Theorem 3.3).

\begin{theorem}\label{3.3}
Let $A$ be a ring of dimension $n\geq 2$ and let $J$ and $J^\prime$ be
two comaximal ideals of $A$ of height $\geq n$ and $n$
respectively. Let $P$ and $P^\prime=Q\op A$ be projective $A$-modules
of rank $n$ and let $\chi: \wedge^n P\iso \wedge^n P^\prime$ be an
isomorphism. Let $\ga : P\surj J\cap J'$ and $\gb : P^\prime \surj J'$
be surjections. Let "bar" denote reduction modulo $J'$ and let $\ol
\ga : \ol P \surj J'/J'^2$ and $\ol \gb: \ol {P^\prime} \surj J'/J'^2$
be surjections induced from $\ga$ and $\gb$ respectively.  Suppose
there exists an isomorphism $\delta : \ol P \iso \ol {P^\prime}$ such
that $\ol \gb \delta =\ol \ga$ and $\wedge^n(\delta) = \ol \chi$. Then
there exists a surjection $\theta : P\surj J$ such that $\theta \ot
A/J=\ga \ot A/J$.
\end{theorem}

\begin{lemma}\label{6.7}
(\cite{BR}, Proposition 6.7) Let $A$ be a ring of dimension $n$ and
let $P,P^\prime$ be stably isomorphic projective $A$-modules of rank
$n$. Then there exists an ideal $J$ of $A$ of height $\geq n$ such
that $J$ is a surjective image of both $P$ and $P^\prime$. Further,
given any ideal $K$ of height $\geq 1$, $J$ can be chosen to be
comaximal with $K$.
\end{lemma}

We state the following result from (\cite{BK}, Proposition 2.11) for later use.

\begin{proposition}\label{BK2.11}
Let $A$ be a ring and let $I$ be an ideal of $A$ of height $n$. Let
$f\in A$ be a non-zerodivisor modulo $I$ and let $P=P_1\op A$ be a
projective $A$-module of rank $n$. Let $\ga:P \ra I$ be a linear map
such that the induced map $\ga_f:P_f \surj I_f$ is a surjection. Then,
there exists $\Psi \in EL(P_f^*)$ such that

$(i)$ $\gb=\Psi(\ga) \in P^*$ and

$(ii)$ $\gb(P)$ is an ideal of $A$ of height $n$ contained in $I$.
\end{proposition}

%%%%%%%%%%%%%%%%%%%%%%%%%%%%%%%%%%%%%%%%%%%%%%%%%%%%%%%%%%%%%%%%%%%%%
\section{Euler class group $E(A,F)$}

Let $A$ be a ring of dimension $n\geq 2$ and let $F=Q\op A$ be a
projective $A$-module of rank $n$.  We define the Euler class group of
$A$ with respect to $F$ as follows:

Let $J$ be an ideal of $A$ of height $n$ such that $J/J^2$ is
generated by $n$ elements. Let $\ga$ and $\gb$ be two surjections from
$F/JF $ to $J/J^2$. We say that $\ga$ and $\gb$ are {\it related} if
there exists an automorphism $\sigma$ of $F/JF$ of determinant $1$
such that $\ga \sigma=\gb$. Clearly, this is an equivalence relation
on the set of all surjections from $F/JF$ to $J/J^2$. Let $[\ga]$
denote the equivalence class of $\ga$. We call $[\ga]$ a {\it local
$F$-orientation} of $J$.

Since $\dim A/J =0$, $\SL_{A/J}(F/JF)=EL(F/JF)$ and therefore, by
(\ref{B-Roy4.1}), the canonical map from $\SL_A(F)$ to
$\SL_{A/J}(F/JF)$ is surjective. Hence, if a surjection $\ga : F/JF
\surj J/J^2$ can be lifted to a surjection $\Delta : F \surj J$, then
so can any other surjection $\gb$ equivalent to $\ga$.

A local $F$-orientation $[\ga]$ is called a {\it global
$F$-orientation} of $J$ if the surjection $\ga$ can be lifted to a
surjection from $F$ to $J$.  From now on, we shall identify a
surjection $\ga$ with the equivalence class $[\ga]$ to which $\ga$
belongs.

Let $\CM$ be a maximal ideal of $A$ of height $n$ and let $\CN$ be a
$\CM$-primary ideal such that $\CN/\CN^2$ is generated by $n$
elements. Let $w_{\CN}$ be a local $F$-orientation of $\CN$. Let $G$
be the free abelian group on the set of pairs $(\CN,w_{\CN})$, where
$\CN$ is a $\CM$-primary ideal and $w_{\CN}$ is a local
$F$-orientation of $\CN$.

Let $J=\cap \CN_i$ be the intersection of finitely many
$\CM_i$-primary ideals, where $\CM_i$ are distinct maximal ideals of
$A$ of height $n$. Assume that $J/J^2$ is generated by $n$ elements
and let $w_J$ be a local $F$-orientation of $J$. Then $w_J$ gives
rise, in a natural way, to local $F$-orientations $w_{\CN_i}$ of
$\CN_i$. We associate to the pair $(J,w_J)$, the element $\sum
(\CN_i,w_{\CN_i})$ of $G$.

Let $H$ be the subgroup of $G$ generated by the set of pairs
$(J,w_J)$, where $J$ is an ideal of $A$ of height $n$ and $w_J$ is a
global $F$-orientation of $J$.

We define the {\it Euler class group} of $A$ with respect to $F$,
denoted by $E(A,F)$, as the quotient group $G/H$.

Let $A$ be a ring of dimension $n$.  Let $P$ and $F=Q\op A$ be
projective $A$-modules of rank $n$ and let $\chi : \wedge^n F \iso
\wedge^n P$ be an isomorphism.  We call $\chi$ a {\it $F$-orientation}
of $P$. To the pair $(P,\chi)$, we associate an element $e(P,\chi)$ of
$E(A,F)$ as follows:

Let $\gl : P\surj J_0$ be a generic surjection of $P$ and let "bar" denote
reduction modulo the ideal $J_0$. Then, we obtain an induced
surjection $\ol \gl : \ol P \surj J_0/J_0^2$. Since $\dim A/J_0=0$,
every projective $A/J_0$-module of constant rank is free. Hence, we
choose an isomorphism $\ol \gamma : F/J_0F \iso P/J_0P$ such that
$\wedge^n (\ol \gamma)=\ol \chi$. Let $w_{J_0}$ be the local
$F$-orientation of $J_0$ given by $\ol \gl \circ \ol \gamma : F/J_0F
\surj J_0/J_0^2$. Let $e(P,\chi)$ be the image in $E(A,F)$ of the
element $(J_0,w_{J_0})$ of $G$. We say that $(J_0,w_{J_0})$ is
obtained from the pair $(\gl,\chi)$. We will show that the assignment
sending the pair $(P,\chi)$ to the element $e(P,\chi)$ of $E(A,F)$ is
well defined.

Let $\mu:P \surj J_1$ be another generic surjection of $P$.  By (\ref{3.0}),
there exists a surjection $\ga(T) : P[T] \surj I$, where $I$ is an
ideal of $A[T]$ of height $n$ with $\ga(0)=\gl$, $I(0)=J_0$,
$\ga(1)=\mu$ and $I(1)=J_1$. Using (\ref{3.1}), we get an ideal $K$ of
$A$ of height $n$ and a local $F$-orientation $w_K$ of $K$ such that
$(I(0),w_{I(0)}) + (K,w_K)=0=(I(1),w_{I(1)})+(K,w_K)$ in
$E(A,F)$. Therefore $(J_0,w_{J_0})=(J_1,w_{J_1})$ in $E(A,F)$. Hence
$e(P,\chi)$ is well defined in $E(A,F)$.

We define the {\it Euler class} of $(P,\chi)$ to be $e(P,\chi)$.

For a projective $A$-module $L$ of rank $1$ and $F=L\op A^{n-1}$, the
following result is proved in (\cite{BR}, Proposition 4.1). Since the
same proof works in our case, we omit the proof.

\begin{proposition}\label{4.1}
Let $A$ be a ring of dimension $n\geq 2$ and let $J,J_1,J_2$ be ideals
of $A$ of height $n$ such that $J$ is comaximal with $J_1$ and
$J_2$. Let $F=Q\op A$ be a projective $A$-module of rank $n$. Assume
that $\ga : F \surj J\cap J_1$ and $\gb : F \surj J\cap J_2$ be
surjections with $\ga\ot A/J=\gb\ot A/J$. Suppose there exists an
ideal $J_3$ of height $n$ such that

$(i)$ $J_3$ is comaximal with $J,J_1$ and $J_2$ and

$(ii)$ there exists a surjection
$\gamma : F \surj J_3\cap J_1$ with $\ga\ot A/J_1=\gamma\ot A/J_1$.

 Then there exists
a surjection $\gl : F \surj J_3\cap J_2$ with $\gl \ot A/J_3=\gamma\ot A/J_3$ and
$\gl\ot A/J_2=\gb\ot A/J_2$.
\end{proposition}

Using (\ref{4.1}, \ref{addition} and \ref{3.3}) and following the
proof of (\cite{BR}, Theorem 4.2), the next result follows.

\begin{theorem}\label{4.2}
Let $A$ be a ring of dimension $n\geq 2$ and let $F=Q\op A$ be a
projective $A$-module of rank $n$.  Let $J$ be an ideal of $A$ of
height $n$ such that $J/J^2$ is generated by $n$ elements. Let $w_J :
F/JF \surj J/J^2$ be a local $F$-orientation of $J$. Suppose that the
image of $(J,w_J)$ is zero in $E(A,F)$. Then $w_J$ is a global
$F$-orientation of $J$.
\end{theorem}

Using (\ref{4.2} and \ref{3.3}) and following the proof of (\cite{BR},
Corollary 4.3), the next result follows.

\begin{corollary}\label{4.3}
Let $A$ be a ring of dimension $n\geq 2$. Let $P$ and $F=Q\op A$ be
projective $A$-modules of rank $n$ and let $\chi : \wedge^n F\iso
\wedge^n P$ be a $F$-orientation of $P$.  Let $J$ be an ideal of $A$
of height $n$ such that $J/J^2$ is generated by $n$ elements and let
$w_J$ be a local $F$-orientation of $J$. Suppose $e(P,\chi)=(J,w_J)$
in $E(A,F)$. Then there exists a surjection $\ga : P\surj J$ such that
$(J,w_J)$ is obtained from $(\ga,\chi)$.
\end{corollary}

Using (\ref{4.2}, \ref{4.3}) and following the proof of (\cite{BR},
Theorem 4.4), the next result follows.

\begin{corollary}\label{4.4}
Let $A$ be a ring of dimension $n\geq 2$. Let $P$ and $F=Q\op A$ be
projective $A$-modules of rank $n$ and let $\chi:\wedge^n F\iso
\wedge^{n} P$ be a $F$-orientation of $P$.  Then $e(P,\chi)=0$ in
$E(A,F)$ if and only if $P$ has a unimodular element.
\end{corollary}

Let $A$ be a ring of dimension $n\geq 2$ and let $F=Q\op A$ be a
projective $A$-module of rank $n$. Let ``bar" denote reduction modulo
the nil radical $N$ of $A$ and let $\ol A=A/N$ and $\ol F=F/NF$. Let
$J$ be an ideal of $A$ of height $n$ with primary decomposition
$J=\cap \CN_i$. Then $\ol J=(J+N)/N$ is an ideal of $\ol A$ of height
$n$ with primary decomposition $\ol J=\cap \ol\CN_i$. Moreover, any
surjection $w_J: F/JF \surj J/J^2$ induces a surjection $\ol w_{\ol J}
: \ol F/\ol {JF} \surj \ol J/\ol {J^2}=(J+N)/(J^2+N)$. Hence, the
assignment sending $(J,w_J)$ to $(\ol J,\ol w_{\ol J})$ gives rise to
a group homomorphism $\Phi : E(A,F) \surj E(\ol A,\ol F)$.

As a consequence of (\ref{4.2}), we get the following result, the
proof of which is same as of (\cite{BR}, Corollary 4.6).

\begin{corollary}\label{4.6}
The homomorphism $\Phi : E(A,F) \ra E(\ol A,\ol F)$ is an isomorphism.
\end{corollary}
%%%%%%%%%%%%%%%%%%%%%%%%%%%%%%%%%%%%%%%%%%%%%%%%%%%%%%%%%%%%%%%%%%%%%%%%%%%%%%%5

\section{Some results on $E(A,F)$}

Let $A$ be a ring of dimension $n\geq 2$ and let $F=Q\op A$ be a
projective $A$-module of rank $n$. Let $J$ be an ideal of $A$ of
height $n$ and let $w_J: F/JF \surj J/J^2$ be a surjection.  Let $\ol
b\in A/J$ be a unit. Then composing $w_J$ with an automorphism of
$F/JF$ of determinant $\ol b$, we get another local $F$-orientation
of $J$, which we denote by $\ol b w_J$. Further, if $w_J$ and $\wt
w_J$ are two local $F$-orientations of $J$, then it is easy to see
that $\wt w_J=\ol b w_J$ for some unit $\ol b\in A/J$.

We recall the following two results from (\cite{BR}, Lemma 2.7 and
2.8) respectively.

\begin{lemma}\label{2.7}
Let $A$ be a ring and let $P$ be a projective $A$-module of rank
$n$. Assume $0\ra P_1 \ra A\op P \by {(b,-\ga)} A \ra 0$ is an exact
sequence. Let $(a_0,p_0)\in A\op P$ be such that
$a_0b-\ga(p_0)=1$. Let $q_i=(a_i,p_i)\in P_1$ for $i=1,\ldots,
n$. Then

$(i)$ the map $\delta: \wedge^n P_1\ra \wedge^n P$ given by $\delta
(q_1\wedge \ldots \wedge q_n)=a_0(p_1\wedge \ldots \wedge p_n) + \sum
_1^n (-1)^i a_i(p_0\wedge \ldots p_{i-1}\wedge p_{i+1} \ldots \wedge
p_n)$ is an isomorphism.

$(ii)$ $\delta(bq_1\wedge \ldots \wedge q_n)=p_1\wedge \ldots \wedge p_n$.
\end{lemma}

\begin{lemma}\label{2.8}
Let $A$ be a ring and let $P$ be a projective $A$-module of rank
$n$. Assume $0\ra P_1 \ra A\op P \by {(b,-\ga)} A \ra 0$ is an exact
sequence. Then

$(i)$ The map $\gb: P_1 \ra A$ given by $\gb(q)=c$, where $q=(c,p)$,
has the property that $\gb(P_1)=\ga(P)$.

$(ii)$ The map $\Phi: P\ra P_1$ given by $\Phi(p)=(\ga(p),bp)$ has the
property that $\gb\circ \Phi=\ga$ and $\delta\circ \wedge^n \Phi$ is a
scalar multiplication by $b^{n-1}$, where $\delta$ is as in
(\ref{2.7}).
\end{lemma}

The following result can be deduced from (\ref{2.7},
\ref{2.8}). Briefly it says that if there exists a projective
$A$-module $P$ of rank $n$ with a $F$-orientation $\chi :\wedge^n
F\iso \wedge^n P$ such that $e(P,\chi)=(J,w_J)$ and if $\ol a \in A/J$
is a unit, then there exists another projective $A$-module $P_1$ with
$[P_1]=[P]$ in $K_0(A)$ and a $F$-orientation $\chi_1 : \wedge^n F
\iso \wedge^n P_1$ of $P_1$ such that $e(P_1,\chi_1)=(J,\ol
{a^{n-1}}w_J)$.

\begin{lemma}\label{5.1}
Let $A$ be a ring of dimension $n\geq 2$. Let $P$ and $F=Q\op A$ be
projective $A$-modules of rank $n$ and let $\chi :\wedge^n F\iso
\wedge^{n} P$ be a $F$-orientation of $P$.  Let $\ga : P \surj J$ be a
generic surjection of $P$ and let $(J,w_J)$ be obtained from
$(\ga,\chi)$. Let $a,b\in A$ with $ab=1$ modulo $J$ and let $P_1$ be
the kernel of the surjection $(b,-\ga) : A\op P \surj A$. Let $\gb :
P_1\surj J$ be as in (\ref{2.8}) and let $\chi_1$ be the
$F$-orientation of $P_1$ given by $\delta^{-1} \chi$, where $\delta$
is as in (\ref{2.7}). Then $(J,\ol {a^{n-1}})$ is obtained from
$(\gb,\chi_1)$.
\end{lemma}

Using the above results and following the proof of (\cite{BR}, Lemmas
5.3, 5.4 and 5.5) respectively, the next three results follows.

\begin{lemma}\label{5.3}
Let $A$ be a ring of dimension $n\geq 2$ and let $F=Q\op A^2$ be a
projective $A$-module of rank $n$. Let $J$ be an ideal of $A$ of
height $n$ and let $w_J : F/JF \surj J/J^2$ be a surjection. Suppose
$w_J$ can be lifted to a surjection $\ga: F \surj J$. Let $\ol a\in
A/J$ be a unit and let $\theta$ be an automorphism of $F/JF$ with
determinant $\ol {a^2}$. Then the surjection $w_J\circ \theta : F/JF
\surj J/J^2$ can be lifted to a surjection $\gamma : F\surj J$.
\end{lemma}

\begin{lemma}\label{5.4}
Let $A$ be a ring of dimension $n\geq 2$ and let $F=Q\op A^2$ be a
projective $A$-module of rank $n$. Let $J$ be an ideal of $A$ of
height $n$ and let $w_J$ be a local $F$-orientation of $J$. Let $\ol
a\in A/J$ be a unit. Then $(J,w_J) =(J,\ol {a^2} w_J)$ in $E(A,F)$.
\end{lemma}

\begin{lemma}\label{5.5}
Let $A$ be a ring of dimension $n\geq 2$ and let $F=Q\op A$ be a
 projective $A$-module of rank $n$. Let $J$ be an ideal of $A$ of
 height $n$ and let $w_J$ be a local $F$-orientation of $J$.  Suppose
 $(J,w_J) \not= 0$ in $E(A,F)$. Then there exists an ideal $J_1$ of
 height $n$ which is comaximal with $J$ and a local $F$-orientation
 $w_{J_1}$ of $J_1$ such that $(J,w_J)+(J_1,w_{J_1})=0$ in $E(A,F)$.
 Further, given any ideal $K$ of $A$ of height $\geq 1$, $J_1$ can be
 chosen to be comaximal with $K$.
\end{lemma}

The following result is similar to (\cite{BR}, Lemma 5.6).

\begin{lemma}\label{5.6}
Let $A$ be an affine domain of dimension $n\geq 2$ over a field $k$
and let $f$ be a non-zero element of $A$. Let $F=Q\op A^2$ be a
projective $A$-module of rank $n$ and let $J$ be an ideal of $A$ of
height $n$ such that $J/J^2$ is generated by $n$ elements. Suppose
that $(J,w_J)\not= 0$ in $E(A,F)$ but the image of $(J,w_J)$ is zero
in $E(A_f,F_f)$. Then there exists an ideal $J_2$ of $A$ of height $n$
such that $(J_2)_f=A_f$ and $(J,w_J)= (J_2,w_{J_2})$ in $E(A,F)$.
\end{lemma}

\begin{proof}
Since $(J,w_J)\not= 0$ in $E(A,F)$, but its image is zero in
$E(A_f,F_f)$, we see that $f$ is not a unit in $A$. By (\ref{5.5}), we
can choose an ideal $J_1$ of height $n$ which is comaximal with $Jf$
such that $(J,w_J)+(J_1,w_{J_1})=0$ in $E(A,F)$.  Since the image of
$(J,w_J)$ is zero in $E(A_f,F_f)$, it follows that the image of
$(J_1,w_{J_1})$ is zero in $E(A_f,F_f)$.

By (\ref{4.2}), there exists a surjection $\ga: F_f \surj (J_1)_f$
such that $\ga \ot A_f/(J_1)_f = (w_{J_1})_f$. Choose a positive
integer $k$ such that $f^{2k}\ga : F \ra J_1$. Since $f$ is a unit
modulo $J_1$, by (\ref{5.4}), $(J_1,w_{J_1})=(J_1,\ol
{f^{2kn}}w_{J_1})$ in $E(A,F)$. By (\ref{BK2.11}), there exists $\Psi
\in EL(F_f^*)$ such that $\gb=\Psi(\ga) \in F^*$ and $\gb(F)\subset
J_1$ is an ideal of height $n$. Thus $\gb(F)=J_1\cap J_2$, where $J_2$
is an ideal of $A$ of height $n$ such that $(J_2)_f=A_f$. Hence
$J_1+J_2=A$.  From the surjection $\gb$, we get
$(J_1,w_{J_1})+(J_2,w_{J_2})=0$ in $E(A,F)$. Since
$(J,w_J)+(J_1,w_{J_1})=0$ in $E(A,F)$, it follows that
$(J,w_J)=(J_2,w_{J_2})$ in $E(A,F)$. This proves the result. $\hfill
\square$
\end{proof}

Using (\ref{4.3}, \ref{5.4} and \ref{5.6}) and following the proof of
(\cite{BR}, Lemma 5.8), the following result can be proved.

\begin{lemma}\label{5.8}
Let $A$ be an affine domain of dimension $n\geq 2$ over a field
$k$. Let $P$ and $F=Q\op A^2$ be projective $A$-modules of rank $n$
with $\wedge^n P\iso \wedge^{n} F$. Let $f$ be a non-zero element of
$A$. Assume that every generic surjection ideal of $P$ is surjective
image of $F$. Then every generic surjection ideal of $P_f$ is
surjective image of $F_f$.
\end{lemma}

Using above results and following the proof of (\cite{BR}, Theorem
5.9), the next result follows.

\begin{theorem}\label{5.9}
Let $A$ be an affine domain of dimension $n\geq 2$ over a real closed
field $k$. Let $P$ and $F=Q\op A^2$ be projective $A$-modules of rank
$n$ with $\wedge^n P\iso \wedge^{n} F$.  Assume that every generic
surjection ideal of $P$ is surjective image of $F$. Then $P$ has a
unimodular element.

In particular, if $L=\wedge^n P$ and every generic surjection ideal of
$P$ is surjective image of $L\op A^{n-1}$, then $P$ has a unimodular
element.
\end{theorem}

%%%%%%%%%%%%%%%%%%%%%%%%%%%%%%%%%%%%%%%%%%%%%%%%%%%%%%%%%%%%%%%%%%%%%%%%%%%%%5

\section{Weak Euler Class Group}

Let $A$ be a ring of dimension $n\geq 2$ and let $F=Q\op A$ be a
projective $A$-module of rank $n$. We define the weak Euler class
group $E_0(A,F)$ of $A$ with respect to $F$ as follows:

Let $\CS$ be the set of ideals $\CN$ of $A$ such that $\CN/\CN^2$ is
generated by $n$ elements, where $\CN$ is $\CM$-primary ideal for some
maximal ideal $\CM$ of $A$ of height $n$. Let $G$ be the free abelian
group on the set $\CS$.

Let $J=\cap \CN_i$ be the intersection of finitely many ideals
$\CN_i$, where $\CN_i$ is $\CM_i$-primary and $\CM_i$'s are distinct
maximal ideals of $A$ of height $n$. Assume that $J/J^2$ is generated
by $n$ elements. We associate to $J$, the element $\sum \CN_i$ of
$G$. We denote this element by $(J)$.

Let $H$ be the subgroup of $G$ generated by elements of the type
$(J)$, where $J$ is an ideal of $A$ of height $n$ which is surjective
image of $F$.

We set $E_0(A,F)=G/H$.

Let $P$ be a projective $A$-module of rank $n$ such that $\wedge^nP
\iso \wedge^{n}F$.  Let $\gl : P\surj J_0$ be a generic surjection of $P$. We
define $e(P)=(J_0)$ in $E_0(A,F)$. We will show that this assignment
is well defined.

Let $\mu : P \surj J_1$ be another generic surjection of $P$.  By
(\ref{3.0}), there exists a surjection $\ga(T): P[T] \surj I$, where
$I$ is an ideal of $A[T]$ of height $n$ with $\ga(0)=\gl$, $I(0)=J_0$,
$\ga(1)=\mu$ and $I(1)=J_1$. Now, as before, using (\ref{3.1}), we see
that $(J_0)=(J_1)$ in $E_0(A,F)$. This shows that $e(P)$ is well
defined.

Note that there is a canonical surjection from $E(A,F)$ to $E_0(A,F)$ obtained by
forgetting the orientations.

We state the following result which follows from (\ref{5.1} and \ref{5.4}).

\begin{lemma}\label{5.1a}
Let $A$ be a ring of even dimension $n$. Let $P$ and $F=Q\op A^2$ be
projective $A$-modules of rank $n$ and let $\chi : \wedge^n F\iso
\wedge^n P$ be a $F$-orientation of $P$. Let $e(P,\chi)=(J,w_J)$ in
$E(A,F)$ and let $\wt w_J$ be another local $F$-orientation of
$J$. Then there exists a projective $A$-module $P_1$ with $[P_1]=[P]$
in $K_0(A)$ and a $F$-orientation $\chi_1$ of $P_1$ such that
$e(P_1,\chi_1)=(J,\wt w_J)$ in $E(A,F)$.
\end{lemma}

\begin{proposition}\label{6.1}
Let $A$ be a ring of even dimension $n$ and let $F=Q\op A^2$ be a
projective $A$-module of rank $n$. Let $J_1$ and $J_2$ be two
comaximal ideals of $A$ of height $n$ and let $J_3=J_1 \cap J_2$. If
any two of $J_1,J_2$ and $J_3$ are surjective images of projective
$A$-modules of rank $n$ which are stably isomorphic to $F$, then so is
the third one.
\end{proposition}

\begin{proof}
$(i)$
Let $P_1$ and $P_2$ be two projective $A$-modules of rank $n$ with
$[P_1]=[P_2]=[F]$ in $K_0(A)$ and let $\psi_1 : P_1\surj J_1$ and
$\psi_2 :P_2\surj J_2$ be two surjections.  Choose $F$-orientations
$\chi_1$ and $\chi_2$ of $P_1$ and $P_2$ respectively.  Then
$e(P_1,\chi_1)=(J_1,w_{J_1})$ and $e(P_2,\chi_2)=(J_2,w_{J_2})$ in
$E(A,F)$.

By (\ref{6.7}), there exists an ideal $J_1^\prime$ of height $n$ which
is surjective image of both $P_1$ and $F$. Hence
$e(P_1,\chi_1)=(J_1,w_{J_1})=(J_1^\prime,w_{J_1^\prime})$ in $E(A,F)$
for some local $F$-orientation $w_{J_1^\prime}$ of
$J_1^\prime$. Similarly, there exists an ideal $J_2^\prime$ of height
$n$ which is surjective image of both $P_2$ and $F$. Hence
$e(P_2,\chi_2)=(J_2,w_{J_2})=(J_2^\prime,w_{J_2^\prime})$ in $E(A,F)$
for some local $F$-orientation $w_{J_2^\prime}$ of
$J_2^\prime$. Further, we may assume that
$J_1^\prime+J_2^\prime=A$. Let
$(J_1,w_{J_1})+(J_2,w_{J_2})=(J_3,w_{J_3})$ in $E(A,F)$.

Let $J_3^\prime=J_1^\prime \cap J_2^\prime$. By addition principle
(\ref{addition}), $J_3^\prime$ is surjective image of $F$ and
$(J_1^\prime,w_{J_1^\prime})+(J_2^\prime,w_{J_2^\prime})
=(J_3^\prime,w_{J_3^\prime})$
in $E(A,F)$. Hence $(J_3^\prime,w_{J_3^\prime})=(J_3,w_{J_3})$. Since
$J_3^\prime$ is surjective image of $F$, by (\ref{5.1a}), there exists
a projective $A$-module $P_3$ with $[P_3]=[F]$ in $K_0(A)$ and a
$F$-orientation $\chi_3$ of $P_3$ such that
$e(P_3,\chi_3)=(J_3^\prime,w_{J_3^\prime})=(J_3,w_{J_3})$ in $E(A,F)$.
By (\ref{4.3}), there exists a surjection $\psi_3 : P_3\surj J_3$ such
that $(\psi_3,\chi_3)$ induces $(J_3,w_{J_3})$. This proves the first
part. \\

$(ii)$ Now assume that $J_1$ and $J_3$ are surjective images of
$P_1^\prime$ and $P_3$ respectively, where $P_1^\prime$ and $P_3$ are
projective $A$-modules of rank $n$ with $[P_1^\prime]=[P_3]=[F]$ in
$K_0(A)$.

Let $e(P_3,\chi_3)=(J_3,w_3)$ for some $F$-orientation $\chi_3$ of
$P_3$ and let $(J_3,w_3)=(J_1,w_1)+(J_2,w_2)$ in $E(A,F)$. Let
$e(P_1^\prime,\chi_1^\prime)=(J_1,w_1^\prime)$ for some
$F$-orientation $\chi_1^\prime$ of $P_1^\prime$.  By (\ref{5.1a}),
there exists a projective $A$-module $P_1$ of rank $n$ with
$[P_1]=[P_1^\prime]$ in $K_0(A)$ and a $F$-orientation $\chi_1$ of
$P_1$ such that $e(P_1,\chi_1)=(J_1,w_1)$ in $E(A,F)$.

By (\ref{6.7}), there exists an ideal $J_4$ of height $n$ which is
surjective image of $F$ and $P_1$ both and is comaximal with $J_2$
such that $e(P_1,\chi_1)=(J_1,w_1)=(J_4,w_4)$.  Write $J_5=J_4 \cap
J_2$. Assume that $(J_4,w_4) + (J_2,w_2)=(J_5,w_5)$ in $E(A,F)$. Then we have
$e(P_3,\chi_3)=(J_3,w_3)=(J_5,w_5)$ in $E(A,F)$.

Since $J_4$ is a surjective image of $F$, we get
$e(F,\chi)=(J_4,\wt w_4)=0$ for some $\chi$. If $(J_4,\wt w_4)+(J_2,w_2)=(J_5,\wt
w_5)$, then $(J_2,w_2)=(J_5,\wt w_5)$.  Since
$e(P_3,\chi_3)=(J_5,w_5)$, by (\ref{5.1a}), there exists a projective
$A$-module $\wt P_3$ of rank $n$ with $[\wt P_3]=[P_3]$ in $K_0(A)$
such that $e(\wt P_3,\wt w_3)=(J_5,\wt w_5)=(J_2,w_2)$. Hence, by
(\ref{4.3}), $J_2$ is a surjective image of $\wt P_3$ which is stably
isomorphic to $F$. This completes the proof.  $\hfill \square$
\end{proof}

\begin{proposition}\label{6.2}
Let $A$ be a ring of even dimension $n$ and let $F=Q\op A^2$ be a
projective $A$-module of rank $n$. Let $J$ be an ideal of $A$ of
height $n$. Then $(J)=0$ in $E_0(A,F)$ if and only if $J$ is a
surjective image of a projective $A$-module of rank $n$ which is
stably isomorphic to $F$.
\end{proposition}

\begin{proof}
Let $J_1$ be an ideal of $A$ of height $n$. Assume that $J_1$ is
surjective image of a projective $A$-module of rank $n$ which is
stably isomorphic to $F$. Assume $(J_1,w_{J_1})$ is a non-zero element
of $E(A,F)$. We will show that there exist height $n$ ideals $J_2$ and
$J_3$ with local $F$-orientations $w_{J_2}$ and $w_{J_3}$ respectively
such that

$(i)$ $J_2,J_3$ are comaximal with any given ideal of height $\geq 1$,

 $(ii)$ $(J_1,w_{J_1})=-(J_2,w_{J_2})=(J_3,w_{J_3})$ in $E(A,F)$ and

 $(iii)$ $J_2,J_3$ are surjective images of projective $A$-modules of
 rank $n$ which are stably isomorphic to $F$.

By (\ref{5.5}), there exists an ideal $J_2$ of height $n$ which is
comaximal with $J_1$ and any given ideal of height $\geq 1$ such that
$(J_1,w_{J_1})+(J_2,w_{J_2})=0$ in $E(A,F)$.  By (\ref{4.2}), $J_1\cap
J_2$ is surjective image of $F$. By (\ref{6.1}), $J_2$ is a surjective
image of a projective $A$-module of rank $n$ which is stably
isomorphic to $F$.

Repeating the above with $(J_2,w_{J_2})$, we get an ideal $J_3$ of
height $n$ which is comaximal with any given ideal of height $\geq 1$
such that $(J_2,w_{J_2})+(J_3,w_{J_3})=0$ in $E(A,F)$. Further, $J_3$
is a surjective image of a projective $A$-module of rank $n$ which is
stably isomorphic to $F$. Thus, we have
$(J_1,w_{J_1})=-(J_2,w_{J_2})=(J_3,w_{J_3})$ in $E(A,F)$. This proves
the above claim.

From the above discussion, we see that given any element $h$ in kernel
 of the canonical map $\Phi :E(A,F) \surj E_0(A,F)$, there exists an
 ideal $\wt J$ of height $n$ such that $\wt J$ is surjective image of
 a projective $A$-module of rank $n$ which is stably isomorphic to $F$
 and $h=(\wt J,w_{\wt J})$ in $E(A,F)$. Moreover, $\wt J$ can be
 chosen to be comaximal with any ideal of height $\geq 1$.

Now assume $(J)=0$ in $E_0(A,F)$. Choose some local $F$-orientation
$w_J$ of $J$. Then $(J,w_J) \in ker(\Phi)$. From previous paragraph,
we get that there exists an ideal $K$ of height $n$ comaximal with $J$
such that $-(J,w_J)=(K,w_K)$ in $E(A,F)$. Further, $K$ is surjective
image of a projective $A$-module which is stably isomorphic to $F$.

Since $(J,w_J)+(K,w_K)=0$ in $E(A,F)$, by (\ref{4.2}), $J\cap K$ is
surjective image of $F$. By (\ref{6.1}), $J$ is surjective image of a
projective $A$-module of rank $n$ which is stably isomorphic to $F$.\\

Conversely, assume that $J$ is surjective image of a projective
$A$-module $P$ of rank $n$ which is stably isomorphic to $F$. Let
$\chi$ be a $F$-orientation of $P$. Then $e(P,\chi)=(J,w_J)$ in
$E(A,F)$. By (\ref{6.7}), there exists an ideal $I$ of height $n$
which is surjective image of $P$ and $F$ both. Then
$e(P,\chi)=(J,w_J)=(I,w_I)$ in $E(A,F)$.  Therefore $(J)=(I)$ in
$E_0(A,F)$ and hence $(J)=0$ in $E_0(A,F)$. This completes the proof.
$\hfill \square$
\end{proof}

\begin{proposition}\label{6.3}
Let $A$ be a ring of even dimension $n$ and let $F=Q\op A^2$ and $P$
 be projective $A$-modules of rank $n$ with $\wedge^n P\iso \wedge^{n}
 F$.  Then $e(P)=0$ in $E_0(A,F)$ if and only if $[P]=[P_1\op A]$ in
 $K_0(A)$ for some projective $A$-module $P_1$ of rank $n-1$.
\end{proposition}

\begin{proof}
Assume that $[P]=[P_1\op A]$ in $K_0(A)$. By (\ref{6.7}), there exists
an ideal $J$ of $A$ of height $ n$ which is surjective image of both
$P$ and $P_1\op A$. Hence $e(P_1\op A,\chi)=(J,w_J)=0 $ in $E(A,F)$,
by (\ref{4.4}). Hence $J$ is surjective image of $F$. By (\ref{6.2}),
$e(P)=(J)=0$ in $E_0(A,F)$.

Conversely, assume that $e(P)=0$ in $E_0(A,F)$. Let $\psi: P\surj J$
be a generic surjection of $P$ and let $e(P,\chi)=(J,w_J)$ in $E(A,F)$ for
some $F$-orientation $\chi$ of $P$. Since $e(P)=(J)=0$ in $E_0(A,F)$,
by (\ref{6.2}), $J$ is surjective image of a projective $A$-module
$P_1$ with $[P_1]=[F]$ in $K_0(A)$. By (\ref{6.7}), there exists an
height $n$ ideal $J_1$ which is surjective image of $P_1$ and $F$
both. Let $e(P_1,\chi_1)=(J,\wt w_J)=(J_1,w_{J_1})$ for some
$F$-orientation $\chi_1$ of $P_1$.

By (\ref{5.1a}), there exists a rank $n$ projective $A$-module $P_2$
with $[P_2]=[P]$ in $K_0(A)$ and a $F$-orientation $\chi_2$ of $P_2$
such that $e(P_2,\chi_2)=(J,\wt w_J)=(J_1,w_{J_1})$ in $E(A,F)$.
Since $J_1$ is a surjective image of $F$, $(J_1,\wt w_{J_1})=0$ in
$E(A,F)$ for some local $F$-orientation $\wt w_{J_1}$ of $J_1$. By
(\ref{5.1a}), there exists a projective $A$-module $P_3$ with
$[P_3]=[P_2]$ in $K_0(A)$ and a $F$-orientation $\chi_3$ of $P_3$ such
that $e(P_3,\chi_3)=(J_1,\wt w_{J_1})=0$ in $E(A,F)$. Hence
$P_3=P_4\op A$, by (\ref{4.4}).  Therefore $[P]=[P_2]=[P_4\op A]$ in
$K_0(A)$. This completes the proof.  $\hfill \square$
\end{proof}

\begin{proposition}\label{6.4}
Let $A$ be a ring of even dimension $n$. Let $P$ and $F=Q\op A^2$ be
projective $A$-modules of rank $n$ with $\wedge^n P\iso \wedge^{n}
F$. Suppose that $e(P)=(J)$ in $E_0(A,F)$, where $J$ is an ideal of
$A$ of height $n$. Then there exists a projective $A$-module $P_1$ of
rank $n$ such that $[P]=[P_1]$ in $K_0(A)$ and $J$ is a surjective
image of $P_1$.
\end{proposition}

\begin{proof}
Since $P/JP$ is free and $J/J^2$ is generated by $n$ elements, we get
a surjection $\ol \psi : P/JP \surj J/J^2$. By (\cite{BR}, Corollary
2.14), we can lift $\ol \psi$ to a surjection $\psi: P \surj J\cap
J_1$, where $J_1$ is an height $n$ ideal comaximal with $J$. Let
$e(P,\chi)=(J,w_J) +(J_1,w_{J_1})$ in $E(A,F)$ for some
$F$-orientation $\chi$ of $P$.

Since $e(P)=(J)=(J\cap J_1) $ in $E_0(A,F)$, $(J_1)=0$ in
$E_0(A,F)$. By (\ref{6.2}), $J_1$ is surjective image of a projective
$A$-module $P_2$ of rank $n$ which is stably isomorphic to $F$. By
(\ref{5.1a}), there exists rank $n$ projective $A$-module $P_3$ with
$[P_2]=[P_3]$ in $K_0(A)$ and a $F$-orientation $\chi_3$ of $P_3$ such
that $e(P_3,\chi_3)=(J_1,w_{J_1})$ in $E(A,F)$.

By (\ref{6.7}), there exists an ideal $J_2$ of height $n$ which is
comaximal with $J$ and is surjective image of $F$ and $P_3$
both. Assume that $e(P_3,\chi_3)=(J_1,w_{J_1})=(J_2,w_{J_2})$ in
$E(A,F)$. Hence $e(P,\chi)=(J,w_J)+(J_2,w_{J_2})=(J\cap J_2,w_{J\cap
J_2})$.  By (\ref{4.3}), there exists a surjection $\phi : P \surj
J\cap J_2$. Since $J_2$ is a surjective image of $F$, we get $(J_2,\wt
w_{J_2})=0$ for some local $F$-orientation $\wt w_{J_2}$ of $J_2$. Let
$(J,w_J)+(J_2,\wt w_{J_2})=(J\cap J_2,\wt w_{J \cap J_2})$. By
(\ref{5.1}), there exists rank $n$ projective $A$-module $P_1$ with
$[P]=[P_1]$ in $K_0(A)$ and $e(P_1,\chi_1)=(J\cap J_2,\wt w_{J\cap
J_2})=(J,w_J)$ in $E(A,F)$ for some $F$-orientation $\chi_1$ of $P_1$. By
(\ref{4.3}), there exists a surjection $\ga : P_1 \surj J$. This
proves the result.  $\hfill \square$
\end{proof}

The proof of the following result is similar to (\cite{BR},
Proposition 6.5), hence we omit it.

\begin{proposition}
Let $A$ be a ring of even dimension $n$ and let $J$ be an ideal of $A$
of height $n$ such that $J/J^2$ is generated by $n$ elements. Let
$F=Q\op A^2$ be a projective $A$-module of rank $n$ and let $\wt w_J:
F/JF \surj J/J^2$ be a surjection. Suppose that the element $(J,\wt
w_J)$ of $E(A,F)$ belongs to the kernel of the canonical homomorphism
$E(A,F)\surj E_0(A,F)$. Then there exists a projective $A$-module
$P_1$ of rank $n$ such that $[P_1]=[F]$ in $K_0(A)$ and
$e(P_1,\chi_1)=(J,\wt w_J)$ in $E(A,F)$ for some $F$-orientation
$\chi_1$ of $P_1$.
\end{proposition}

\section{Application}

Let $A$ be a ring of dimension $n\geq 2$ and let $L$ be a projective
 $A$-module of rank $1$.  We will define a map $\Delta$ from $E(A,L)$
 (defined by Bhatwadekar and Raja Sridharan \cite{BR} and is same as
 $E(A,L\op A^{n-1})$)
 to $E(A,F)$, where $F=Q\op A$ is a projective $A$-module of rank $n$
 with determinant $L$. Let $w_J : L/JL\op (A/J)^{n-1} \surj J/J^2$ be
 a surjection. Since $\dim A/J=0$, $Q/JQ$ is isomorphic to $L/JL \op
 (A/J)^{n-2}$. Choose an isomorphism $\theta : Q/JQ \iso L/JL \op
 (A/J)^{n-2}$ of determinant one.  Let $\wt w_J=w_J \circ (\theta,id)
 : Q/JQ \op A/J \surj J/J^2$ be the surjection.

Assume that $w_J$ can be lifted to a surjection $\Phi: L\op A^{n-1}
\surj J$. Write $\Phi=(\Phi_1,a)$. We may assume that $\Phi_1(L\op
A^{n-2})=K$ is an ideal of height $n-1$. Further, we may assume that
the isomorphism $\theta: Q/JQ \iso L/JL \op (A/J)^{n-2}$ is induced
from an isomorphism $\theta^\prime : Q/KQ \iso L/KL\op (A/K)^{n-2}$
(i.e.  $\theta^\prime \ot A/J=\theta$).

Let $(\Phi_2,a):Q\op A \ra J=(K,a)$ be a lift of $\wt w_J$. Then
$\Phi_2 \ot A/K : Q/KQ \surj K/K^2$ is a surjection. Let $\phi_2 :
Q\ra K$ be a lift of $\Phi_2\ot A/K$. Then $\phi_2(Q)+K^2=K$. Hence,
there exists $e\in K^2$ with $e(1-e)\in \phi_2(Q)$ such that
$\phi_2(Q)+Ae=K$. Now it is easy to check that $\phi_2(Q)+Aa=
\phi_2(Q) + (e+(1-e)a)A=K+Aa=J$ and $(\phi_2,e+(1-e)a):Q\op A\surj J$
is a lift of $\wt w_J$.

Hence, we have shown that if $w_J$ can be lifted to a surjection from
$L\op A^{n-1} \surj J$, then $\wt w_J$ can be lifted to a surjection
from $Q\op A$ to $J$. Further, if we choose different isomorphism
$\theta_1 : Q/JQ\op A/J \iso L/JL \op (A/J)^{n-1}$ of determinant one
and $w_1=w_J\circ \theta_1 : Q/JQ\op A/J \surj J/J^2$, then $\wt w_J$
and $w_1$ are connected by an element of $EL(Q/JQ\op A/J)$. Hence, if
we define $\Delta : E(A,L)\ra E(A,F)$ by $\Delta(w_J)=\wt w_J$, then
this map is well defined. It is easy to see that $\Delta$ is a group
homomorphism.

Similarly, we can define a map $\Delta_1 : E(A,F)\ra E(A,L)$ and it is
easy to show that $\Delta\circ \Delta_1=id$ and $\Delta_1\circ \Delta
=id$. Hence, we get the following interesting result:

\begin{theorem}
 Let $A$ be a ring of dimension $n\geq 2$. Let $L$ and $F=Q\op A$ be
 projective $A$-modules of rank $1$ and $n$ respectively with
 $\wedge^n F\iso L$. Then $E(A,L)$ is isomorphic to $E(A,F)$.
\end{theorem}

Let $J$ be an ideal of $A$ of height $n$ such that $J/J^2$ is
generated by $n$ elements. Further assume that there exists a
surjection $\ga : L\op A^{n-1} \surj J$. We will show that $J$ is also
a surjective image of $F=Q\op A$. Let $w_J$ be the local
$L$-orientation of $J$ induced from $\ga$. Then $(J,w_J)=0$ in
$E(A,L)$. Hence $\Delta(J,w_J)=(J,\wt w_J)=0$ in $E(A,F)$. Hence, by
(\ref{4.2}), $J$ is a surjective image of $F$.

We define the map $\wt \Delta : E_0(A,L) \ra E_0(A,F)$ by $(J) \mapsto
(J)$. The above discussion shows that $\wt \Delta$ is well
defined. Similarly, we can define a map $\wt \Delta_1 : E_0(A,F)\ra
E_0(A,L)$ such that $\wt \Delta \circ \wt \Delta_1 =id$ and $\wt
\Delta_1\circ \wt \Delta =id$. Thus we get the following interesting
result:
 
 \begin{theorem}
Let $A$ be a ring of dimension $n\geq 2$. Let $L$ and $F=Q\op A$ be
projective $A$-modules of rank $1$ and $n$ respectively with $\wedge^n
F\iso L$. Then $E_0(A,L)$ is isomorphic to $E_0(A,F)$.
\end{theorem}

Since, by (\cite{BR}, 6.8), $E_0(A,L)$ is canonically isomorphic to
$E_0(A,A)$, we get the surprising result that $E_0(A,F)$ is
canonically isomorphic to $E_0(A,A^n)$ for any projective $A$-module
$F=Q\op A$ of rank $n$.

We end with the following result which follows from (\ref{6.2}).

\begin{proposition}
Let $A$ be a ring of even dimension $n$ and let $J$ be an ideal of $A$
of height $n$ such that $J/J^2$ is generated by $n$ elements. Let $L$
and $P$ be projective $A$-modules of rank $1$ and $n$ respectively
such that $P$ is stably isomorphic to $L\op A^{n-1}$.  Then $J$ is
surjective image of $P$ if and only if given any projective $A$-module
$Q$ of rank $n-2$ with determinant $L$, there exists a projective
$A$-module $P_1$ which is stably isomorphic to $Q\op A^2$ such that
$J$ is surjective image of $P_1$.
\end{proposition}

{}

\end{document}